\documentclass[11pt]{article}

\usepackage{graphicx}
\usepackage{amsmath, amsthm, amssymb, mathtools, txfonts}
\usepackage{enumitem}
\usepackage[english]{babel}
\usepackage{aliascnt}
\usepackage{tikz}
\usetikzlibrary{decorations.markings}
\usetikzlibrary{decorations.pathmorphing}
\usepackage{hyperref}
\usepackage[capitalize]{cleveref}
\usepackage[final]{showlabels}
\usepackage{multirow}

\def\ds{\displaystyle}

\newcommand\A[1]{\mathrm{A}_{#1}}

\newcommand\Path[1][]{\mathrm{Path}_{#1}}
\newcommand\Walk[1][]{\mathrm{Walk}_{#1}}

\newcommand\vertex[3]{\node (#1) at #2[circle, draw, fill=lightgray!100, inner sep=0pt, minimum width=4pt, label=#3] {};}
\newcommand\arc[4][]{\draw[line width=2, white, shorten >=.1cm, shorten <=.1cm](#2)--(#3);\draw[->, decoration={markings,mark=at position 1 with {\arrow[scale=2]{>}}}, postaction={decorate}, shorten >=0.4pt, #1] (#2) -- node[above]{#4} (#3);}
\newcommand\edge[3][]{\draw[#1] (#2) to node[above]{} (#3);}

\def\sgn{\mathrm{sgn}}

\def\R{\mathbb{R}}
\def\Q{\mathbb{Q}}
\def\N{\mathbb{N}}
\def\Z{\mathbb{Z}}
\def\det{\mathrm{det}}
\def\Pf{\mathrm{Pf}}
\def\LE{\mathrm{LE}}
\def\Ed{\mathrm{Ed}}
\def\V{\mathrm{V}}
\def\wt{\mathrm{wt}}

\renewcommand{\subset}{\subseteq}

\def\emp{\varnothing}

\newcommand{\transpose}[1]{#1^{\top}}
\newcommand{\set}[1]{\{1,\dots,#1\}}

\newcommand{\symgroup}[1]{\mathrm{S}_{#1}}

\newcommand{\bd}{\partial}
\newcommand{\intr}{\mathrm{int}}

\newtheorem{theorem}{Theorem}[section] % numbered by section
\crefname{theorem}{Theorem}{Theorems}

\newtheorem{lemma}[theorem]{Lemma}
\crefname{lemma}{Lemma}{Lemmas}

\crefname{proposition}{Proposition}{Propositions}

\newtheorem{corollary}[theorem]{Corollary}
\crefname{corollary}{Corollary}{Corollaries}

\crefname{conjecture}{Conjecture}{Conjectures}

\theoremstyle{definition} % styled differently: not italicized

\crefname{definition}{Definition}{Definitions}

\newtheorem{remark}[theorem]{Remark}
\crefname{remark}{Remark}{Remarks}

\newtheorem{example}[theorem]{Example}
\crefname{example}{Example}{Examples}

\crefname{figure}{Figure}{Figures}

\begin{document}
\title{A principle for converting Lindstr\"om-type lemmas
to Stembridge-type theorems, with applications to walks, groves, and alternating flows}
\author{Owen Biesel\footnote{Carleton College, Northfield, MN. Email: obiesel@carleton.edu
\newline\newline
\emph{Keywords:} Pfaffian, loop-erased walk, grove, alternating flow}}
\maketitle

\abstract{We prove that Fomin's generalization of Lindstr\"om's lemma for paths on acyclic directed graphs to walks on general directed graphs also generalizes a theorem of Stembridge in the same way. Moreover, we show that whenever a family of operations satisfies a Lindstr\"om-type determinant relation, a related family of operations satisfies a Stembridge-type Pfaffian relation. We give example applications to Kenyon and Wilson's work on groves and to Talaska's work on alternating flows.}

\tableofcontents

\section{Introduction}

In \cite{Lindstrom}, Bernt Lindstr\"om wrote down a connection between determinants and families of disjoint paths in directed graphs, which we will call Lindstr\"om's lemma and review in \cref{background}. Ira Gessel and Xavier Viennot noted in \cite{Gessel-Viennot} that his proof works only for \emph{acyclic} directed graphs, but later, Sergey Fomin successfully generalized Lindstr\"om's lemma to arbitrary directed graphs in \cite{Fomin}, using walks in place of paths, with a modified notion of disjointness using a technique of Gregory Lawler's called loop-erasure.
Meanwhile, in \cite{Stembridge} John Stembridge found a connection between families of disjoint paths in directed acyclic graphs and Pfaffians of certain skew-symmetric matrices. 
We can organize these three results into a table, as shown below:
\begin{center}
\begin{tabular}{c|cc}
 & Paths on & Walks on \\
 & acyclic & general \\
 & directed & directed \\
 & graphs & graphs \\
 \hline
 Determinant & Lindstr\"om's & Fomin's \\ 
 relation: & lemma & theorem \\
  &   &   \\
  Pfaffian & Stembridge's &  \\%\multirow{2}{*}{??} \\
   relation: & theorem & 
 \end{tabular}
\end{center}

The empty space in the lower-right corner of the table would be a Pfaffian relation for directed graphs that are not necessarily acyclic, and it is the first goal of this paper to show that the same modifications Fomin made to Lindstr\"om's lemma also suffice to generalize Stembridge's theorem to arbitrary directed graphs.

We will prove this result by first noting a more general principle, which we call the determinant-to-Pfaffian principle, giving a way of producing a Pfaffian identity similar to Stembridge's theorem whenever we have a determinant identity similar to Lindstr\"om's lemma. As examples, we give a proof of Stembridge's theorem that uses fewer hypotheses than in Stembridge's original formulation, deduce the version of Stembridge's theorem for general directed graphs from Fomin's theorem, and produce new Stembridge-type identities for groves on undirected graphs and alternating flows on planar circular networks directly from analogues to Lindstr\"om's lemma in those cases.  A summary of the main results of this paper are shown in the table below:

\begin{center}
\begin{tabular}{c|cccc}
 & Paths on & Walks on & Groves on & Alternating \\
 & acyclic & general & undirected & flows on \\
 & directed & directed & graphs with & planar circular  \\
 & graphs & graphs & boundary & networks\\
 \hline
 Determinant & Lindstr\"om's & Fomin's & \multirow{2}{*}{\cref{grove-determinant}} & \multirow{2}{*}{\cref{flow-determinant}}\\ 
 relation: & lemma & theorem & \\
 $\downarrow$ & $\downarrow$ & $\downarrow$ & $\downarrow$ & $\downarrow$\\
  Pfaffian & Stembridge's &  \multirow{2}{*}{\cref{fomin-stembridge}} &  \multirow{2}{*}{\cref{grove-pfaffian}} & \multirow{2}{*}{\cref{flow-pfaffian}}\\
   relation: & theorem &  & 
 \end{tabular}
 \end{center}
 
We begin by reviewing Lindstr\"om's lemma and Stembridge's theorem in \cref{background} (the readers already familiar with these results may skim these sections to acquaint themselves with our notation). In \cref{general-directed} we review the terminology necessary to state both Fomin's theorem and our generalization of Stembridge's theorem to arbitrary directed graphs, \cref{fomin-stembridge}. In \cref{main-section} we state and prove the determinant-to-Pfaffian principle (\cref{main-thm}), and use it to prove the generalization of Stembridge's theorem. Finally, in \cref{undirected,flows} we prove Lindstr\"om-type lemmas for groves on undirected graphs and flows on planar circular networks, and deduce Stembridge-type theorems using the determinant-to-Pfaffian principle.

The author would like to thank Pavlo Pylyavskyy for many helpful conversations throughout the research and writing process, including for alerting the author to the question of whether Stembridge's theorem could be generalized along the lines of Fomin's theorem, and for pointing out the work of Talaska on alternating flows.
 
\section{Lindstr\"om's lemma and Stembridge's theorem}\label{background}

In this section, $G$ will be an acyclic directed graph with vertex set $\V(G)$ and edge set $\Ed(G)$.  We say that such a graph $G$ is \emph{weighted} if to each edge $e\in\Ed(G)$ is assigned a weight $\wt(e)$ in some commutative ring, often taken to be $\R$, the polynomial ring $\Z[e: e\in\Ed(G)]$, or the formal power series ring $\Z[[e: e\in\Ed(G)]]$.

 In the following, we will leave the weight ring implicit and assume that every sum we write down converges, either through a finiteness assumption or by working in the ring of formal power series if necessary.

We begin by reviewing Lindstr\"om's lemma, the most basic link between determinants and families of pairwise disjoint paths. Given two vertices $a,b\in \V(G)$, we define $\Path(a,b)$ as the set of directed paths $a\to b$ in $G$, and $P(a,b)$ as the sum $\sum_{p\in\Path(a,b)} \wt(p)$, where the weight $\wt(p)$ of a path $p$ is the product of the weights of its edges.

We extend this terminology to tuples $\mathbf{a}=(a_1,\dots,a_k)$ and $\mathbf{b} = (b_1,\dots,b_k)$ of vertices by defining $\Path[k](\mathbf{a},\mathbf{b})$ as the set of families of paths $\mathbf{p}=(p_1,\dots,p_k)$ such that each $p_i$ is a path from $a_i$ to $b_i$ and if $i\neq j$ then $p_i$ and $p_j$ are disjoint (i.e.\ have no vertices in common). We define $P_k(\mathbf{a},\mathbf{b})$ as the sum of the weights of these path families:
\[P_k(\mathbf{a},\mathbf{b}) = \sum_{\mathbf{p}\in\Path[k](\mathbf{a},\mathbf{b})} \wt(\mathbf{p}),\]
where the weight of a family of paths $\mathbf{p}=(p_1,\dots,p_k)$ is the product of the path weights: $\wt(\mathbf{p}) = \wt(p_1)\dots \wt(p_k)$. We also define $\tilde P_k(\mathbf{a},\mathbf{b})$ as the signed sum
\[\tilde P_k(\mathbf{a},\mathbf{b}) = \sum_{\sigma\in\symgroup{k}}\sgn(\sigma) P_k(\mathbf{a},\mathbf{b}_\sigma),\]
where $\mathbf{b}_\sigma$ is the permuted tuple $(b_{\sigma(1)},b_{\sigma(2)},\dots,b_{\sigma(k)})$.

\begin{theorem}[Lindstr\"om's lemma, Theorem 1 in \cite{Gessel-Viennot}]\label{Lindstrom}
 Let $G$ be a weighted acyclic directed graph. For any natural number $k$ and pair of $k$-tuples $\mathbf{a}=(a_1,\dots,a_k)$ and $\mathbf{b}=(b_1,\dots,b_n)$ of vertices of $G$, we have
 \[\tilde P_k(\mathbf{a},\mathbf{b}) = \det\bigl(P(a_i,b_j)\bigr)_{i,j=1}^k.\]
\end{theorem}

\begin{example}
 For example, in the weighted directed acyclic graph shown in \cref{figure-DAG}, we have two paths from vertex $2$ to vertex $4$, one using the edges labeled $c$ and $e$, and one using the single edge labeled $f$, so $P(2,4) = ce+f$. \cref{figure-lindstrom} shows that if $\mathbf{a}=(1,2)$ and $\mathbf{b}=(3,4)$, then $\tilde P_2(\mathbf{a},\mathbf{b}) = abf-ade$. We also have $\det\begin{pmatrix} P(1,3) & P(1,4) \\ P(2,3) & P(2,4)\end{pmatrix} = \det\begin{pmatrix} ab & ae \\ bc+d & ce+f\end{pmatrix} = abf - ade$, as Lindstr\"om's lemma predicts.
\end{example}

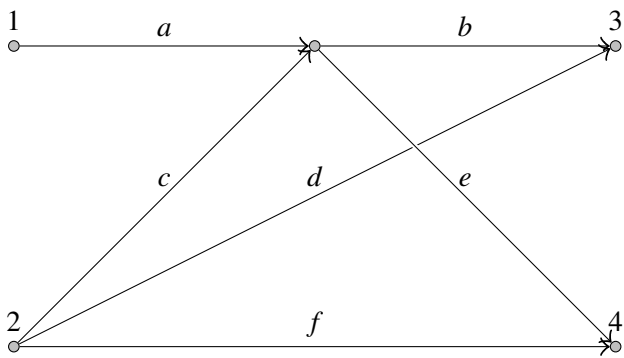
\begin{figure}[h]
\[\begin{tikzpicture}[scale=4]
 \vertex{a1}{(0,1)}{$1$}
 \vertex{a2}{(0,0)}{$2$}
 \vertex{i1}{(1,1)}{}
 \vertex{b1}{(2,1)}{$3$}
 \vertex{b2}{(2,0)}{$4$}
 \arc{a1}{i1}{$a$}
 \arc{i1}{b1}{$b$}
 \arc{a2}{i1}{$c$}
 \arc{a2}{b1}{$d$}
 \arc{i1}{b2}{$e$}
 \arc{a2}{b2}{$f$}
\end{tikzpicture}\]
\caption{
A directed acyclic graph with five vertices and six edges. Each edge is weighted by a formal variable. 
}\label{figure-DAG}\end{figure}

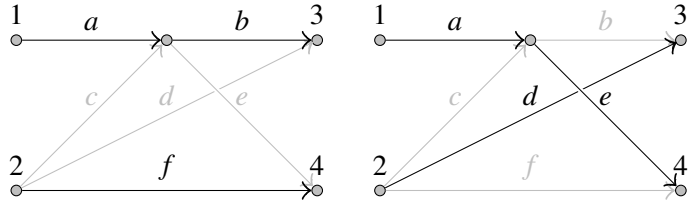
\begin{figure}[h]
\[\begin{tikzpicture}[scale=2]
 \vertex{a1}{(0,0)}{$2$}
 \vertex{a2}{(0,1)}{$1$}
 \vertex{i1}{(1,1)}{}
 \vertex{b1}{(2,0)}{$4$}
 \vertex{b2}{(2,1)}{$3$}
 \arc[lightgray]{a1}{b2}{$d$}
 \arc[lightgray]{a1}{i1}{$c$}
 \arc[lightgray]{i1}{b1}{$e$}
 \arc{a2}{i1}{$a$}
 \arc{a1}{b1}{$f$}
 \arc{i1}{b2}{$b$}
\end{tikzpicture}
\quad
\begin{tikzpicture}[scale=2]
 \vertex{a1}{(0,0)}{$2$}
 \vertex{a2}{(0,1)}{$1$}
 \vertex{i1}{(1,1)}{}
 \vertex{b1}{(2,0)}{$4$}
 \vertex{b2}{(2,1)}{$3$}
 \arc[lightgray]{a1}{b1}{$f$}
 \arc[lightgray]{a1}{i1}{$c$}
 \arc[lightgray]{i1}{b2}{$b$}
 \arc{a1}{b2}{$d$}
 \arc{a2}{i1}{$a$}
 \arc{i1}{b1}{$e$}
\end{tikzpicture}\]
\caption{The two terms in $\tilde P_2(\mathbf{a},\mathbf{b})$ are $abf$ and $-ade$.}
\label{figure-lindstrom}
\end{figure}

\begin{remark}\label{acyclic-necessary}
 The content of Lindstr\"om's lemma is that we need only sum over \emph{disjoint} families of paths; expanding the determinant $\det\bigl(P(a_i,b_j)\bigr)_{i,j=1}^n$ yields the defining expression of $\tilde P_k(\mathbf{a},\mathbf{b})$ but with general families of paths $p_i:a_i\to b_{\sigma(i)}$ instead of disjoint families. The proof goes by constructing a sign-reversing involution on the collection of non-disjoint families of paths, so that only the disjoint families remain after cancellation.
\end{remark}

\renewcommand\arc[4][]{\draw[line width=2, white, shorten >=.1cm, shorten <=.1cm](#2)--(#3);\draw[->, #1] (#2) -- node[above]{#4} (#3);}

Like Lindstr\"om's lemma, Stembridge's theorem also concerns a matrix whose entries are weights of path families, but interprets its \emph{Pfaffian} instead of its determinant. The Pfaffian of a skew-symmetric matrix is a canonical square root of its determinant, and is an integer polynomial in its above-diagonal entries. Each term in $\Pf(A)$ corresponds to a perfect matching on the set of rows (or columns) of $A$, and the term corresponding to a matching is the product of all entries $a_{ij}$ for which $i$ is matched with $j$, with a sign determined by whether the number of ``crossings'' in the matching is even or odd; see \cref{pfaffian}.
\begin{figure}[h]
\[\begin{array}{rcl}\det\begin{pmatrix*}
 0 & a & b & \hphantom{-}c\\
 -a & 0 & d & e \\
 -b & -d & 0 & f\\
 -c & -e & -f & 0
\end{pmatrix*} & = &(af - be + cd)^2 \\ \\
\Pf\begin{pmatrix*}[r]
 0 & a & b & \hphantom{-}c\\
 -a & 0 & d & e \\
 -b & -d & 0 & f\\
 -c & -e & -f & 0
\end{pmatrix*}& = & af - be + cd
\end{array}\qquad
\begin{tikzpicture}[baseline={([yshift=-.5ex]current bounding box.center)}] \vertex{11}{(1,0)}{-90:1}
 \vertex{12}{(2,0)}{-90:2}
 \vertex{13}{(3,0)}{-90:3}
 \vertex{14}{(4,0)}{-90:4}
 \draw (11) to [bend left] node[above]{$a$} (12);
 \draw (13) to [bend left] node[above]{$f$} (14); 
 
 \vertex{21}{(1,-1.5)}{-90:1}
 \vertex{22}{(2,-1.5)}{-90:2}
 \vertex{23}{(3,-1.5)}{-90:3}
 \vertex{24}{(4,-1.5)}{-90:4}
 \draw (21) to [bend left] node[above]{$b$} (23);
 \draw (22) to [bend left] node[above]{$e$} (24);
 
 \vertex{31}{(1,-3)}{-90:1}
 \vertex{32}{(2,-3)}{-90:2}
 \vertex{33}{(3,-3)}{-90:3}
 \vertex{34}{(4,-3)}{-90:4}
 \draw (31) to [bend left=45] node[above]{$c$} (34);
 \draw (32) to [bend left] node[above]{$d$} (33);
\end{tikzpicture}\]
\caption{A $4\times 4$ skew-symmetric matrix with its determinant and Pfaffian. Each term in the Pfaffian corresponds to a perfect matching on $\{1,2,3,4\}$.}
\label{pfaffian}
\end{figure}

Stembridge's theorem also refers to a notion called \emph{compatibility}: two ordered subsets $A$ and $B$ of a directed graph $G$ are called $G$-\emph{compatible} if, for all choices of vertices $a<a'$ in $A$ and $b<b'$ in $B$ and all pairs of paths $p: a\to b'$ and $q: a'\to b$, there must be nontrivial intersection between $p$ and $q$. This can be understood as a kind of planarity condition on $A$ and $B$ (see \cref{planar-compatible}), although the notion is more general.

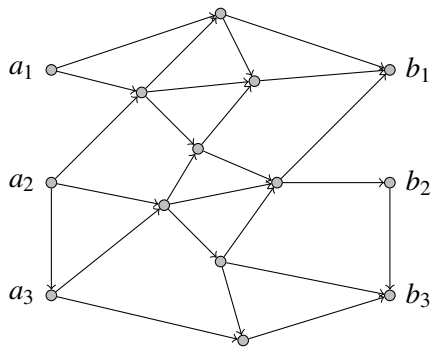
\begin{figure}[h]
\[\begin{tikzpicture}[scale=1.5]
 \vertex{a1}{(0,3)}{180:$a_1$}
 \vertex{a2}{(0,2)}{180:$a_2$}
 \vertex{a3}{(0,1)}{180:$a_3$}
 \vertex{b1}{(3,3)}{0:$b_1$}
 \vertex{b2}{(3,2)}{0:$b_2$}
 \vertex{b3}{(3,1)}{0:$b_3$}
 \vertex{i1}{(1.5,3.5)}{}
 \vertex{i2}{(.8,2.8)}{}
 \vertex{i3}{(1.8,2.9)}{}
 \vertex{i4}{(1.3,2.3)}{}
 \vertex{i5}{(1,1.8)}{}
 \vertex{i6}{(2,2)}{}
 \vertex{i7}{(1.5,1.3)}{}
 \vertex{i8}{(1.7,.6)}{}
 
 \arc{a1}{i1}{}
 \arc{i1}{b1}{}
 \arc{a1}{i2}{}
 \arc{i2}{i1}{}
 \arc{i1}{i3}{}
 \arc{i3}{b1}{}
 \arc{a2}{i2}{}
 \arc{i2}{i3}{}
 \arc{i4}{i3}{}
 \arc{i2}{i4}{}
 \arc{a2}{i5}{}
 \arc{i5}{i4}{}
 \arc{i4}{i6}{}
 \arc{i6}{b1}{}
 \arc{i6}{b2}{}
 \arc{i5}{i6}{}
 \arc{i7}{i6}{}
 \arc{a2}{a3}{}
 \arc{a3}{i5}{}
 \arc{i5}{i7}{}
 \arc{a3}{i8}{}
 \arc{i7}{i8}{}
 \arc{i8}{b3}{}
 \arc{i7}{b3}{}
 \arc{b2}{b3}{}
\end{tikzpicture}
\]
\caption{A directed acyclic graph $G$ in which the ordered sets $A=\{a_1,a_2,a_3\}$ and $B=\{b_1,b_2,b_3\}$ are $G$-compatible. In general, if $G$ is planar and the vertices $a_1,\dots,a_m, b_k,\dots,b_1$ appear in order around a single face of $G$ (in this case, the outer face), then $A=\{a_1,\dots,a_m\}$ and $B=\{b_1,\dots,b_k\}$ are $G$-compatible.}
\label{planar-compatible}
\end{figure}

\begin{theorem}[Stembridge's theorem, Theorem 3.1 in \cite{Stembridge}]\label{Stembridge} 
 Let $G$ be a weighted acyclic directed graph, and fix an ordered set of vertices $B\subset \V(G)$. For all ordered tuples of vertices $\mathbf{a} = (a_1,\dots,a_k)$ that are $G$-compatible with $B$, define $Q_{k}(\mathbf{a})$ as the sum  of $P_k(\mathbf{a},\mathbf{b})$ over all tuples $\mathbf{b} = (b_1,\dots,b_k)$ with $b_1<b_2<\dots<b_k$ in $B$.
 Then for even $k$,
\[Q_k(a_1,\dots,a_k) = \Pf[Q_2(a_i,a_j)]_{1\leq i<j\leq k}.\]
\end{theorem}
See \cref{stembridge-picture} for a visualization of Stembridge's theorem.
\begin{figure}[h]
\[\begin{tikzpicture}[scale=0.75, thick]
 \vertex{a1}{(0,5)}{}
 \vertex{a2}{(0,4)}{180:$a_i$}
 \vertex{a3}{(0,3)}{}
 \vertex{a4}{(0,2)}{180:$a_j$}
 \vertex{a5}{(0,1)}{}
 \vertex{a6}{(0,0)}{}
 \draw (4,5.5) to node[right]{$B$} (4,-.5);
 \draw[->, decorate, decoration={random steps,segment length=3pt,amplitude=1pt}] (a2) to[bend right=10] (4, 4.5);
 \draw[->, decorate, decoration={random steps,segment length=3pt,amplitude=1pt}] (a4) to[bend left=10] (4, 1);
\end{tikzpicture}
\qquad\qquad
\begin{tikzpicture}[scale=0.75, thick]
 \vertex{a1}{(0,5)}{180:$a_1$}
 \vertex{a2}{(0,4)}{180:$a_2$}
 \vertex{a3}{(0,3)}{180:$a_3$}
 \vertex{a4}{(0,2)}{180:$a_4$}
 \vertex{a5}{(0,1)}{180:$\vdots$}
 \vertex{a6}{(0,0)}{180:$a_k$}
 \draw (4,5.5) to node[right]{$B$} (4,-.5);
 \draw[->, decorate, decoration={random steps,segment length=3pt,amplitude=1pt}] (a1) to[bend right=10] (4, 5.1);
 \draw[->, decorate, decoration={random steps,segment length=3pt,amplitude=1pt}] (a2) to[bend right=10] (4, 4.5);
 \draw[->, decorate, decoration={random steps,segment length=3pt,amplitude=1pt}] (a3) to[bend left=10] (4, 2.8);
 \draw[->, decorate, decoration={random steps,segment length=3pt,amplitude=1pt}] (a4) to[bend left=10] (4, 1.6);
 \draw[->, decorate, decoration={random steps,segment length=3pt,amplitude=1pt}] (a5) to[bend left=10] (4, 1);
 \draw[->, decorate, decoration={random steps,segment length=3pt,amplitude=1pt}] (a6) to[bend left=10] (4, -.2);
\end{tikzpicture}\]
\caption{For each pair of vertices $(a_i,a_j)$, we calculate the sum of all weights of disjoint path families from $(a_i,a_j)$ to $B$. The Pfaffian of the skew-symmetric matrix whose above-diagonal entries are these $\binom{k}{2}$ numbers is the sum of the weights of disjoint path families from \emph{all} of the $a_i$ to $B$.}
\label{stembridge-picture}
\end{figure}
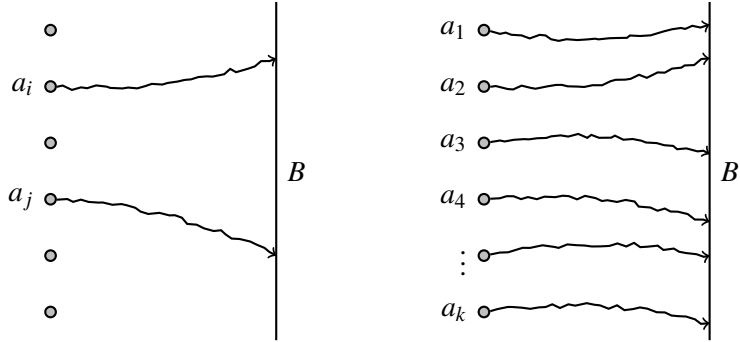

\section{Walks on general directed graphs}\label{general-directed}

Lindstr\"om originally published in \cite[Lemma 1]{Lindstrom} an erroneous proof of \cref{Lindstrom} without the assumption that $G$ be acyclic; if $G$ is not acyclic, then the proof's involution recipe might take a family of paths and produce paths with cycles in them, that is, not paths but walks. (A correct proof of \cref{Lindstrom} in the acyclic case may be found in \cite{Gessel-Viennot}.) For graphs with directed cycles, Lindstr\"om's lemma may fail: in the directed graph shown in \cref{graph-with-directed-cycle}, with $\mathbf{a}=(1,2)$ and $\mathbf{b} = (3,4)$, we have $\tilde P_2(\mathbf{a},\mathbf{b}) = (ab)(ceg)$, but 
 \[\det\begin{pmatrix}
  P(1,3) & P(1,4) \\ P(2,3) & P(2,4)
 \end{pmatrix} = \det\begin{pmatrix}
  ab & adeg \\ bcef & ceg
 \end{pmatrix} = (ab)(ceg)(1 - def).
\]

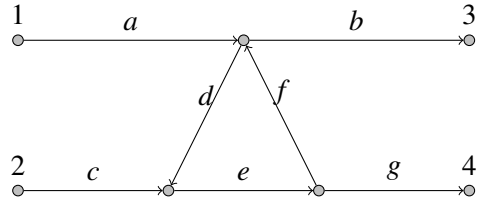
\begin{figure}[h]
\[\begin{tikzpicture}[scale=2]
 \vertex{a1}{(0,1)}{$1$}
 \vertex{a2}{(0,0)}{$2$}
 \vertex{i1}{(1.5,1)}{}
 \vertex{i2}{(1,0)}{}
 \vertex{i3}{(2,0)}{}
 \vertex{b1}{(3,1)}{$3$}
 \vertex{b2}{(3,0)}{$4$}
 \arc{a1}{i1}{$a$}
 \arc{i1}{b1}{$b$}
 \arc{a2}{i2}{$c$}
 \arc{i1}{i2}{$d$}
 \arc{i2}{i3}{$e$}
 \arc{i3}{i1}{$f$}
 \arc{i3}{b2}{$g$}
\end{tikzpicture}\]
\caption{A graph containing a directed cycle for which Lindstr\"om's lemma fails.}
\label{graph-with-directed-cycle}
\end{figure}

A correct generalization of Lindstr\"om's lemma to not-necessarily acyclic directed graphs is Fomin's theorem, which we review now.

In this section, $G$ will be a weighted directed graph that is not necessarily acyclic. Since the sign-reversing involution used in the proof  of Lindst\"om's lemma might produce a family of walks when given a family of paths, one might wonder whether using walks in place of paths might provide a generalization of Lindstr\"om's lemma to not-necessarily-acyclic graphs. Fomin showed in 2001 that the answer is yes, provided that one use an asymmetric form of disjointness involving \emph{loop-erasure}---the loop-erasure $\LE(w)$ of a walk $w$ is formed by following $w$ and erasing any loops as they form (see \cite{Fomin,Lawler} and \cref{loop-erase}).

\begin{figure}[h]
\[\begin{tikzpicture}[baseline=0.25cm]
 \draw[->, rounded corners=10pt](0,0) -- (1.5,1.5) -- (2.5,.5) -- (1.5,-.5) -- (0,1) -- (.5,1.5) -- (2.25,-.25) -- (1.5,-1) -- (.75,-.25) -- (2.5,1.5);
\end{tikzpicture}\ \rightsquigarrow\ 
\begin{tikzpicture}[baseline=0.25cm]
\path[->, rounded corners=10pt](0,0) -- (1.5,1.5) -- (2.5,.5) -- (1.5,-.5) -- (0,1) -- (.5,1.5) -- (2.25,-.25) -- (1.5,-1) -- (.75,-.25) -- (2.5,1.5);
 \draw[->, rounded corners=10pt](0,0) -- (1.5,1.5) -- (2.5,.5) -- (1.5,-.5) -- (.4,.6);
\end{tikzpicture}
\begin{tikzpicture}[baseline=0.25cm]
 \draw[rounded corners=10pt](0,0) -- (.5,.5);
 \draw[rounded corners=10pt, dotted](.5,.5) -- (1.5,1.5) -- (2.5,.5) -- (1.5,-.5)--(.5,.5);
 \draw[->, rounded corners=10pt](.5,.5)-- (0,1) -- (.5,1.5) -- (2.25,-.25) -- (1.5,-1) -- (.75,-.25) -- (1.6,.6);
\end{tikzpicture}
\begin{tikzpicture}[baseline=0.25cm]
 \draw[rounded corners=10pt](0,0) -- (.5,.5);
 \draw[rounded corners=10pt, dotted](.5,.5) -- (1.5,1.5) -- (2.5,.5) -- (1.5,-.5)--(.5,.5);
 \draw[rounded corners=10pt](.5,.5)-- (0,1) -- (.5,1.5) -- (1.5,.5);
 \draw[rounded corners=10pt, dotted](1.5,.5) -- (2.25,-.25) -- (1.5,-1) -- (.75,-.25) -- (1.5,.5);
 \draw[->, rounded corners=10pt](1.5,.5) -- (2.5,1.5);
\end{tikzpicture}\]
\caption{A walk and its loop erasure. As we traverse the walk, we delete each cycle as it appears, leaving a path.
}\label{loop-erase}
\end{figure}
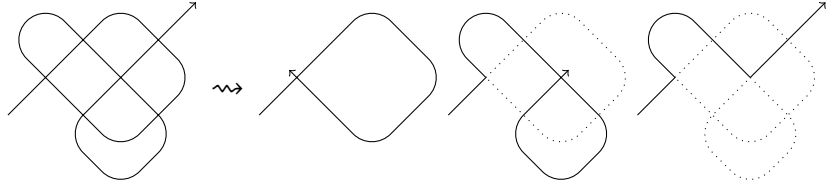

To set up Fomin's theorem, we make the following definitions. Given two vertices $a$ and $b$ of $G$, we let $\Walk(a,b)$ be the set of all walks $w:a\to b$ in $G$, and $W(a,b)$ the sum of these walks' weights, where the weight of a walk is the product of the weights of its edges.

For tuples of vertices $\mathbf{a}=(a_1,\dots,a_k)$ and $\mathbf{b}=(b_1,\dots,b_k)$, we let $\Walk[k](\mathbf{a},\mathbf{b})$ be the collection of all families of walks $(w_1,\dots,w_k)$ where
\begin{itemize}
 \item each $w_i$ is a walk from $a_i$ to $b_i$, and
 \item if $i<j$ then $w_j$ is vertex-disjoint from $\LE(w_i)$.
\end{itemize}
We let $\tilde W_k(\mathbf{a},\mathbf{b})$ be the signed sum of the walk family weights as before:

\[\tilde W_k(\mathbf{a},\mathbf{b}) = \sum_{\sigma\in\symgroup{k}}\sgn(\sigma)\sum_{\mathbf{w}\in\Walk[k](\mathbf{a},\mathbf{b}_\sigma)} \wt(\mathbf{w}),\]
where the weight of a family of walks $\mathbf{w}=(w_1,\dots,w_k)$ is the product of the weights of the walks $w_i$.

\begin{theorem}[Fomin's theorem, Theorem 6.1 in \cite{Fomin}]\label{fomin}
  Let $G$ be a weighted directed graph, not necessarily acyclic. For any natural number $k$ and pair of $k$-tuples $\mathbf{a}=(a_1,\dots,a_k)$ and $\mathbf{b}=(b_1,\dots,b_n)$ of vertices of $G$, we have
 \[\tilde W_k(\mathbf{a},\mathbf{b}) = \det\bigl(W(a_i,b_j)\bigr)_{i,j=1}^k.\]
\end{theorem}

\begin{example}
 Returning to the non-acyclic directed graph from \cref{graph-with-directed-cycle}, we find that each walk from $1$ or $2$ to $3$ or $4$ can traverse the $def$ cycle any number of times, introducing a factor of $1+def +(def)^2+\dots = 1/(1-def)$. Therefore
 \begin{align*}
  \det\begin{pmatrix}
  W(1,3) & W(1,4) \\ W(2,3) & W(2,4)
 \end{pmatrix} &= \det\begin{pmatrix}
  \dfrac{ab}{1-def} & \dfrac{adeg}{1-def} \\ \dfrac{bcef}{1-def} & \dfrac{ceg}{1-def}
 \end{pmatrix}\\ 
 &= \frac{(ab)(ceg)(1 - def)}{(1-def)^2}\\
 &= \frac{(ab)(ceg)}{1-def}.
 \end{align*}
 On the other hand, $\tilde W_2((1,2), (3,4))$ sums over all walk families $w_1:1\to 3$ and $w_2:2\to 4$ such that $w_2$ doesn't intersect $\LE(w_1)$. Then $w_1$ can traverse the $def$ cycle an arbitrary number of times, but $w_2$ must be the path $ceg$. Thus the weights of all of these walk families sum to $\left(\dfrac{ab}{1-def}\right)(ceg)$, which agrees with our earlier determinant calculation in accordance with Fomin's theorem.
\end{example}

We now present a version of Stembridge's theorem for not-necessarily acyclic graphs:

\begin{theorem}\label{fomin-stembridge}
 Let $G$ be a weighted directed graph, and fix an ordered subset $B$ of the vertices of $G$. For each tuple $\mathbf{a}=(a_1,\dots,a_k)$ of vertices of $G$, define $\tilde Q_k(\mathbf{a})$ by
 \[\tilde Q_k(\mathbf{a}) = \sum_{\substack{\mathbf{b}=(b_1,\dots,b_k)\in B^k:\\b_1<\dots<b_k}} \tilde W_k(\mathbf{a},\mathbf{b}).\]
 Then if $k$ is even, we have
 \[\tilde Q_k(\mathbf{a}) = \Pf\bigl(\tilde Q_2(a_i,a_j)\bigr)_{i,j=1}^k.\]
\end{theorem}

Note that in the case of acyclic graphs, all walks are paths and this theorem reduces to a stronger form of Stembridge's theorem where we do not have to assume $A$ and $B$ are compatible.

We will deduce \cref{fomin-stembridge} from \cref{fomin} as a consequence of the \emph{determinant-to-Pfaffian principle} in the next section.

\section{The determinant-to-Pfaffian principle}\label{main-section}

\begin{theorem}[The determinant-to-Pfaffian principle]\label{main-thm}
 Let $A$ and $B$ be sets with $B$ finite, $R$ a commutative ring, and let $\{\tilde C_k: A^k\times B^k\to R\}_{k\in\N}$ be a family of functions that satisfies the determinant relation
 \begin{equation}\label{determinant-relation}
 \tilde C_k(\mathbf{a},\mathbf{b}) = \det\bigl(\tilde C_1(a_i, b_j)\bigr)_{i,j=1}^k
 \end{equation}
 for all $k\in\N$ and for all $\mathbf{a}=(a_1,\dots,a_k)$ in $A^k$ and $\mathbf{b}=(b_1,\dots,b_k)$ in $B^k$.
 
 Fix an ordering of $B$ and define a family of functions $\{\tilde R_{k}:A^{k}\to R\}_{k\in\N}$ by
 \[\tilde R_{k}(\mathbf{a}) = \sum_{\substack{\mathbf{b}=(b_1,\dots,b_{k})\in B^{k}:\\b_1<b_2<\dots<b_{k}}}\tilde C_{k}(\mathbf{a},\mathbf{b}).\]
 Then the family $\{\tilde R_{k}\}_{k\in\N}$ satisfies the Pfaffian relation
 \begin{equation}\label{pfaffian-relation}
 \tilde R_{k}(\mathbf{a}) = \Pf\bigl(\tilde R_2(a_i,a_j)\bigr)_{i,j=1}^{k}
 \end{equation}
 for all even $k$ and for all $\mathbf{a}=(a_1,\dots,a_{k})$ in $A^{k}$.
\end{theorem}
%
%\begin{remark} The family $\tilde C_k$ is completely and freely determined from $\tilde C_1$, so every function $A\times B\to R$ yields such a family of $\tilde C_k$ and $\tilde R_{k}$. The usefulness comes when $\tilde C_k$ can be interpreted on its own.
%\end{remark}

\begin{example}[Alternate proof of Stembridge's theorem]
 Let $G$ be a weighted acyclic directed graph, and let $A$ and $B$ be $G$-compatible ordered subsets of $\V(G)$. If we let $\tilde C_k=\tilde P_k = P_k$, we find that the determinant relation \eqref{determinant-relation} holds by Lindstrom's lemma. Therefore defining $\tilde R_k: A^k\to R$ as 
  \[\tilde R_{k}(\mathbf{a}) = \sum_{\substack{\mathbf{b}=(b_1,\dots,b_{k})\in B^{k}:\\b_1<b_2<\dots<b_{k}}}\tilde C_{k}(\mathbf{a},\mathbf{b}),\]
  we find that if $a_1<\dots<a_k$ then $\tilde R_k(a_1,\dots,a_k) = Q_k(a_1,\dots,a_k)$. So for $k$ even, the Pfaffian relation
  \[Q_k(a_1,\dots,a_k) = \Pf\bigl[Q_2(a_i,a_j)\bigr]_{1\leq i<j\leq k}\]
  holds, as claimed in Stembridge's theorem.
\end{example}

\begin{example}[Proof of \cref{fomin-stembridge}]
If $G$ is an arbitrary weighted directed graph, with vertex subsets $A,B\subseteq \V(G)$ with $B$ finite, then the determinant relation \eqref{determinant-relation} holds with $\tilde C_k = \tilde W_k$ by Fomin's theorem, and the resulting functions $\tilde R_k$ that satisfy the Pfaffian relation \eqref{pfaffian-relation} are exactly the $\tilde Q_k$, as claimed in \cref{fomin-stembridge}.
\end{example}

In \cref{undirected,flows}, we will have two more example applications of the determinant-to-Pfaffian principle, to groves on weighted undirected graphs and to alternating flows on planar circular networks. For the remainder of this section we focus on proving \cref{main-thm}. The proof uses the following two lemmas:

\begin{lemma}[Theorem 3.2 of \cite{Ishikawa-Wakayama}]
 Let $M$ be a skew-symmetric $m\times m$ matrix, and let $D$ be a $k\times m$ matrix with $k$ even. Then
 \[\Pf(DM\transpose{D}) = \sum_{1\leq j_1<\dots<j_k\leq m}\Pf\bigl(M_{\{j_1,\dots,j_k\}}\bigr)\det\bigl(D_{\{1,\dots,k\}}^{\{j_1,\dots,j_k\}}\bigr),\]
 where for subsets $I\subset\set{k}$ and $J\subset\set{m}$ we use the notation $D_I^J$ for the submatrix of $D$ consisting of rows indexed by $I$ and columns indexed by $J$, and by $M_J$ we mean the skew-symmetric principal submatrix $M_J^J$.
\end{lemma}

\begin{lemma}\label{basic-pfaffian}
 For each natural number $n$, let $M_{n}$ be the $n\times n$ skew-symmetric matrix with each above-diagonal entry equal to $1$.
 Then for all $n$, we have $\Pf(M_{2n}) = 1$.
\end{lemma}
\begin{proof}
 This follows immediately by induction on $n$ from the recursive formula for the Pfaffian of a $2n\times 2n$ matrix: The base case is $\Pf(M_0)=1$, and for larger $n$, we have
 \begin{align*}
 \Pf(M_{2n}) &= \sum_{j=2}^{2n} (-1)^j a_{1j} \Pf\bigl((M_{2n})_{\{1,\dots,2n\}\setminus\{1,j\}}\bigr)\\
  &= \sum_{j=2}^{2n} (-1)^j (1) \Pf(M_{2n-2}) = \sum_{j=2}^{2n} (-1)^j = 1.\qedhere
 \end{align*}
\end{proof}

\begin{proof}[Proof of \cref{main-thm}]
 Let $\mathbf{a}=(a_1,\dots,a_{k})\in A^{k}$ with $k$ even. We will construct a $k\times k$ matrix which, on the one hand, has $ij$th entry equal to $\tilde R_2(a_i,a_j)$, and on the other hand, has Pfaffian equal to $\tilde R_k(\mathbf{a})$.
 
 Let $D$ be a matrix whose rows are indexed by $\{1,\dots,k\}$ and whose columns are indexed by elements of $B$: the $ib$th entry of $D$ is defined to be
 \[D_{ib} = \tilde C_1(a_i,b).\]
We also define $M$ to be the skew-symmetric matrix whose rows and columns are both indexed by $B$, where
\[M_{bb'} = \begin{cases}\hphantom{-}1&\text{ if }b<b'\\
-1&\text{ if }b>b'\\
\hphantom{-}0 &\text{ if }b=b'.\end{cases}\]
Then the matrix product $DM\transpose{D}$ is skew-symmetric of size $k\times k$, and its $ij$th entry is
\begin{align*}
 (DM\transpose{D})_{ij} &= \sum_{b,b'\in B} \tilde C_1(a_i,b) M_{bb'} \tilde C_1(a_j,b')\\
 &= \sum_{b<b'\in B} \tilde C_1(a_i,b)\tilde C_1(a_j,b') - \tilde C_1(a_i,b')\tilde C_1(a_j,b)\\
 &= \sum_{b<b'\in B}\det\left(\begin{matrix}\tilde C_1(a_i,b) & \tilde C_1(a_i, b')\\ \tilde C_1(a_j,b) & \tilde C_1(a_j,b')\end{matrix}\right)\\
 &= \sum_{b<b'\in B} \tilde C_2((a_i, a_j), (b,b'))\\
 &= \tilde R_2(a_i,a_j).
\end{align*}
All that remains, then, is to show that $\Pf(DM\transpose{D}) = \tilde R_k(\mathbf{a})$. We use Ishikawa and Wakayama's minor summation formula:
\[\Pf(DM\transpose{D}) = \sum_{b_1 < \dots < b_k\in B} \det\bigl(D_{\{1,\dots,k\}}^{\{b_1,\dots,b_k\}}\bigr)\Pf\bigl(M_{\{b_1,\dots,b_k\}}\bigr).\]
Now for any tuple $\mathbf{b}=(b_1,\dots,b_k)$ in $B^k$, the matrix $D_{\{1,\dots,k\}}^{\{b_1,\dots,b_k\}}$ equals $\bigl(\tilde C_1(a_i,b_j)\bigr)_{i,j=1}^k$ and has determinant $\tilde C_k(\mathbf{a},\mathbf{b})$. Also, if $b_1<b_2<\dots,b_k$ then $M_{b_1,\dots,b_k}$ is a matrix in the form described by \cref{basic-pfaffian}, so its Pfaffian is $1$ since $k$ is even. Thus
\[\Pf(DM\transpose{D}) = \sum_{\substack{\mathbf{b} = (b_1,\dots,b_k)\in B^k:\\b_1 < \dots < b_k}} \tilde C_k(\mathbf{a},\mathbf{b}) = \tilde R_k(\mathbf{a}).\qedhere\]
\end{proof}

\section{Groves on undirected graphs-with-boundary}\label{undirected}
 For another application of the determinant-to-Pfaffian principle, let $G$ be a finite and weighted \emph{un}directed graph, together with a designated partition of the vertices $\V(G) = \intr V\sqcup \bd V$ into \emph{interior vertices} and \emph{boundary vertices}; this makes $G$ a \emph{graph-with-boundary}. We also assume that every connected component of $G$ contains at least one boundary vertex. A \emph{grove} is a spanning forest for $G$ such that every component tree contains at least one boundary vertex; every grove induces a partition on $\bd V$ based on which boundary vertices are in the same trees.
 We will denote the set of groves inducing the singleton partition on $\bd V$ by $\mathrm{Tree}(\bd V)$; see \cref{tree}.
 
 \begin{figure}[h]
 \[\begin{tikzpicture}[scale=1.5, thick]
 \vertex{a1}{(0,3)}{, fill=black}
 \vertex{a2}{(0,2)}{, fill=black}
 \vertex{a3}{(0,1)}{, fill=black}
 \vertex{b1}{(3,3)}{, fill=black}
 \vertex{b2}{(3,2)}{, fill=black}
 \vertex{b3}{(3,1)}{, fill=black}
 \vertex{i1}{(1.5,3.5)}{}
 \vertex{i2}{(.8,2.8)}{}
 \vertex{i3}{(1.8,2.9)}{}
 \vertex{i4}{(1.3,2.3)}{}
 \vertex{i5}{(1,1.8)}{}
 \vertex{i6}{(2,2)}{}
 \vertex{i7}{(1.5,1.3)}{}
 \vertex{i8}{(1.7,.6)}{}
 
 \edge[lightgray]{a1}{i1}
 \edge[lightgray]{i1}{b1}
 \edge{a1}{i2}
 \edge{i1}{i2}
 \edge[lightgray]{i1}{i3}
 \edge[lightgray]{i3}{b1}
 \edge[lightgray]{a2}{i2}
 \edge{i2}{i3}
 \edge[lightgray]{i3}{i4}
 \edge[lightgray]{i2}{i4}
 \edge{a2}{i5}
 \edge[lightgray]{i5}{i4}
 \edge{i4}{i6}
 \edge{i6}{b1}
 \edge[lightgray]{i6}{b2}
 \edge[lightgray]{i5}{i6}
 \edge[lightgray]{i6}{i7}
 \edge[lightgray]{a2}{a3}
 \edge[lightgray]{a3}{i5}
 \edge[lightgray]{i5}{i7}
 \edge[lightgray]{a3}{i8}
 \edge[lightgray]{i7}{i8}
 \edge{i8}{b3}
 \edge{i7}{b3}
 \edge[lightgray]{b2}{b3}
\end{tikzpicture}
\]
\caption{A grove in $\mathrm{Tree}(\bd V)$. (Vertices in $\bd V$ are colored black.)}
\label{tree}
\end{figure}
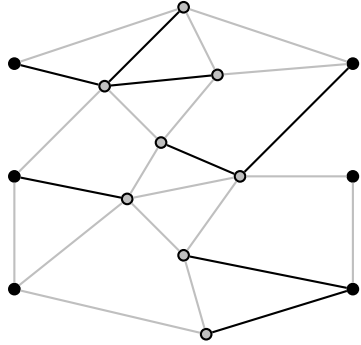
Furthermore, given two disjoint tuples of boundary vertices $\mathbf{a}=(a_1,\dots,a_k)$ and $\mathbf{b}=(b_1,\dots,b_k)$, let $\mathrm{Grove}(\mathbf{a},\mathbf{b})$ be the set of groves inducing the partition 
  with parts $\{a_i,b_i\}$ for $1\leq i\leq k$, and singletons for all other boundary vertices (see \cref{groves}). (If either $\mathbf{a}$ or $\mathbf{b}$ contains repeated vertices, we define $\mathrm{Grove}_k(\mathbf{a},\mathbf{b})=\emp$.)

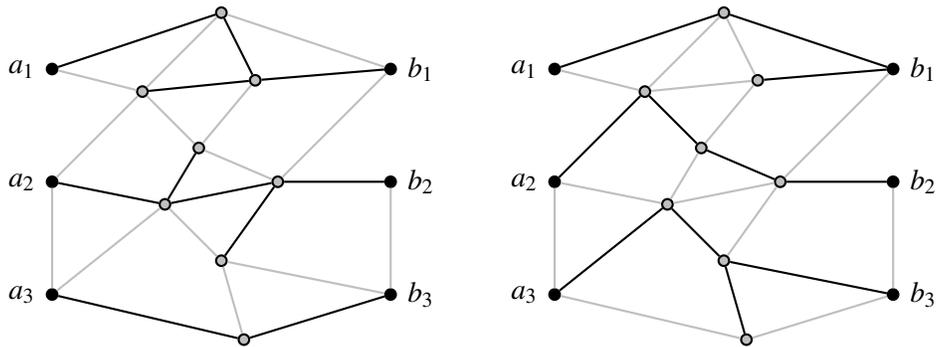
\begin{figure}[h]
 \[\begin{tikzpicture}[scale=1.5, thick]
 \vertex{a1}{(0,3)}{180:$a_1$, fill=black}
 \vertex{a2}{(0,2)}{180:$a_2$, fill=black}
 \vertex{a3}{(0,1)}{180:$a_3$, fill=black}
 \vertex{b1}{(3,3)}{0:$b_1$, fill=black}
 \vertex{b2}{(3,2)}{0:$b_2$, fill=black}
 \vertex{b3}{(3,1)}{0:$b_3$, fill=black}
 \vertex{i1}{(1.5,3.5)}{}
 \vertex{i2}{(.8,2.8)}{}
 \vertex{i3}{(1.8,2.9)}{}
 \vertex{i4}{(1.3,2.3)}{}
 \vertex{i5}{(1,1.8)}{}
 \vertex{i6}{(2,2)}{}
 \vertex{i7}{(1.5,1.3)}{}
 \vertex{i8}{(1.7,.6)}{}
 
 \edge{a1}{i1}
 \edge[lightgray]{i1}{b1}
 \edge[lightgray]{a1}{i2}
 \edge[lightgray]{i1}{i2}
 \edge{i1}{i3}
 \edge{i3}{b1}
 \edge[lightgray]{a2}{i2}
 \edge{i2}{i3}
 \edge[lightgray]{i3}{i4}
 \edge[lightgray]{i2}{i4}
 \edge{a2}{i5}
 \edge{i5}{i4}
 \edge[lightgray]{i4}{i6}
 \edge[lightgray]{i6}{b1}
 \edge{i6}{b2}
 \edge{i5}{i6}
 \edge{i6}{i7}
 \edge[lightgray]{a2}{a3}
 \edge[lightgray]{a3}{i5}
 \edge[lightgray]{i5}{i7}
 \edge{a3}{i8}
 \edge[lightgray]{i7}{i8}
 \edge{i8}{b3}
 \edge[lightgray]{i7}{b3}
 \edge[lightgray]{b2}{b3}
\end{tikzpicture}
\qquad \begin{tikzpicture}[scale=1.5, thick]
 \vertex{a1}{(0,3)}{180:$a_1$, fill=black}
 \vertex{a2}{(0,2)}{180:$a_2$, fill=black}
 \vertex{a3}{(0,1)}{180:$a_3$, fill=black}
 \vertex{b1}{(3,3)}{0:$b_1$, fill=black}
 \vertex{b2}{(3,2)}{0:$b_2$, fill=black}
 \vertex{b3}{(3,1)}{0:$b_3$, fill=black}
 \vertex{i1}{(1.5,3.5)}{}
 \vertex{i2}{(.8,2.8)}{}
 \vertex{i3}{(1.8,2.9)}{}
 \vertex{i4}{(1.3,2.3)}{}
 \vertex{i5}{(1,1.8)}{}
 \vertex{i6}{(2,2)}{}
 \vertex{i7}{(1.5,1.3)}{}
 \vertex{i8}{(1.7,.6)}{}
 
 \edge{a1}{i1}
 \edge{i1}{b1}
 \edge[lightgray]{a1}{i2}
 \edge[lightgray]{i1}{i2}
 \edge[lightgray]{i1}{i3}
 \edge{i3}{b1}
 \edge{a2}{i2}
 \edge[lightgray]{i2}{i3}
 \edge[lightgray]{i3}{i4}
 \edge{i2}{i4}
 \edge[lightgray]{a2}{i5}
 \edge[lightgray]{i5}{i4}
 \edge{i4}{i6}
 \edge[lightgray]{i6}{b1}
 \edge{i6}{b2}
 \edge[lightgray]{i5}{i6}
 \edge[lightgray]{i6}{i7}
 \edge[lightgray]{a2}{a3}
 \edge{a3}{i5}
 \edge{i5}{i7}
 \edge[lightgray]{a3}{i8}
 \edge{i7}{i8}
 \edge[lightgray]{i8}{b3}
 \edge{i7}{b3}
 \edge[lightgray]{b2}{b3}
\end{tikzpicture}
\]
\caption{Two groves in $\mathrm{Grove}(\mathbf{a},\mathbf{b})$.}
\label{groves}
\end{figure}

Then define
\begin{align*}
Z_{\bd V} &\coloneqq \sum_{g\in\mathrm{Tree}(\bd V)}\wt(g)\\
G_k(\mathbf{a},\mathbf{b}) &\coloneqq \frac{1}{Z_{\bd V}} \sum_{g\in\mathrm{Grove}_k(\mathbf{a},\mathbf{b})}\wt(g)\\
\tilde G_k(\mathbf{a},\mathbf{b}) &\coloneqq \sum_{\sigma\in\symgroup{k}}\sgn(\sigma) G_k(\mathbf{a},\mathbf{b}_\sigma),
\end{align*}
  where the weight of a grove is the product of the weights of its edges. (The sum $Z_{\bd V}$ must therefore be invertible in the ambient commutative ring $R$ we are using. One can either work in $\mathbb{R}$ and assume all the edge weights are positive, or we can simply work in the field of rational functions $\Q(e : e\in \Ed(G))$ where each edge is weighted by its own transcendental variable.)

\begin{theorem}\label{grove-determinant}
 For any disjoint subsets $A,B$ of $\bd V$, the family of functions $\{\tilde G_k : A^k\times B^k\to R\}_{k\in\N}$ satisfies the determinant relation \eqref{determinant-relation}. In other words, for all disjoint tuples of boundary vertices $\mathbf{a}=(a_1,\dots,a_k)$ and $\mathbf{b}=(b_1,\dots,b_k)$, we have $\tilde G_k(\mathbf{a},\mathbf{b}) = \det\bigl(\tilde G_1(a_i,b_j)\bigr)_{i,j=1}^k$.
\end{theorem}

The proof reduces to a formula due to Curtis and Morrow, but before we prove \cref{grove-determinant}, we need a lemma which one can regard as a generalization of Kirchhoff's matrix-tree theorem. For $G$ a finite weighted undirected graph, define its \emph{Kirchhoff matrix} $K$ to be the symmetric matrix with rows and columns indexed by the vertices of $G$ whose $ij$th entry for $i\neq j$ is minus the sum of the weights of edges between $i$ and $j$, and whose $ii$th entry is the sum of all weights of edges from $i$ to any other vertex.

\begin{lemma}
 Let $G$ be a finite weighted undirected graph with Kirchhoff matrix $K$. For any choice of partition $\V(G) = \bd V \sqcup \intr V$ into boundary and interior vertices, we have
 \[Z_{\bd V} = \det(K_{\intr V}^{\intr V}),\]
 the determinant of the principal proper submatrix of $K$ indexed by the interior vertices of $G$.
\end{lemma}

\begin{proof}
 This lemma follows easily from the so-called All Minors Matrix-Tree Theorem (see \cite{Chaiken}, for example), but there is also a simple argument based on Kirchhoff's matrix-tree theorem for weighted graphs, which is precisely this lemma in the case of exactly one boundary vertex. For a general number of boundary vertices, form a new weighted graph by gluing all the old boundary vertices into one new boundary vertex. This doesn't change $K_{\intr V}^{\intr V}$, since all the interior vertices have the same patterns of incident edges. Furthermore, it doesn't change $Z_{\bd V}$ either, since a collection of edges forms a spanning forest with one component per boundary vertex if and only if those edges form a spanning tree when the boundary vertices are all identified (see \cref{forest-tree}). Then since the lemma holds for the new graph, it holds for the original graph as well.
\end{proof}

\begin{figure}[h]
\[\begin{tikzpicture}[baseline=-.5cm, thick]
\vertex{a1}{(-.25,.25)}{,fill=black}
\vertex{a2}{(.75,1)}{,fill=black}
\vertex{a3}{(2.5,0)}{,fill=black}
\vertex{i1}{(1.25,0)}{}
\vertex{i2}{(-.5,-1)}{}
\vertex{i3}{(.5,-1)}{}
\vertex{i4}{(1.5,-1)}{}
\vertex{i5}{(2.5,-1.5)}{}
\vertex{i6}{(0,-2)}{}
\vertex{i7}{(1,-2)}{}
\vertex{i8}{(2,-2.5)}{}

\edge[lightgray]{a1}{a2}
\edge{i1}{a2}
\edge{i2}{a1}
\edge{i3}{a1}
\edge{i4}{a3}
\edge[lightgray]{i1}{a3}
\edge[lightgray]{i5}{a3}
\edge[lightgray]{i3}{i1}
\edge[lightgray]{i4}{i1}
\edge[lightgray]{i3}{i2}
\edge[lightgray]{i6}{i2}
\edge{i6}{i3}
\edge[lightgray]{i4}{i3}
\edge{i5}{i4}
\edge[lightgray]{i7}{i4}
\edge{i7}{i5}
\edge{i8}{i5}
\edge[lightgray]{i7}{i6}
\edge[lightgray]{i8}{i7}
\end{tikzpicture}\qquad\rightsquigarrow\qquad
\begin{tikzpicture}[baseline=-.5cm, thick]
\vertex{a1}{(1,1.5)}{,fill=black}
\vertex{a2}{(1,1.5)}{,fill=black}
\vertex{a3}{(1,1.5)}{,fill=black}
\vertex{i1}{(1.25,0)}{}
\vertex{i2}{(-.5,-1)}{}
\vertex{i3}{(.5,-1)}{}
\vertex{i4}{(1.5,-1)}{}
\vertex{i5}{(2.5,-1.5)}{}
\vertex{i6}{(0,-2)}{}
\vertex{i7}{(1,-2)}{}
\vertex{i8}{(2,-2.5)}{}

%\edge[lightgray, bend left]{a1}{a2}
\draw[lightgray, rounded corners=3pt] (a1) -- (.7,.85) -- (.5,1) -- (a2);
\edge[bend left]{i1}{a2}
\edge[bend left]{i2}{a1}
\edge[bend left=45]{i3}{a1}
\edge[bend right=45]{i4}{a3}
\edge[lightgray, bend right=30]{i1}{a3}
\edge[lightgray, bend right=45]{i5}{a3}
\edge[lightgray]{i3}{i1}
\edge[lightgray]{i4}{i1}
\edge[lightgray]{i3}{i2}
\edge[lightgray]{i6}{i2}
\edge{i6}{i3}
\edge[lightgray]{i4}{i3}
\edge{i5}{i4}
\edge[lightgray]{i7}{i4}
\edge{i7}{i5}
\edge{i8}{i5}
\edge[lightgray]{i7}{i6}
\edge[lightgray]{i8}{i7}
\end{tikzpicture}\]
\caption{A spanning forest with one tree per boundary vertex is the same data as a spanning tree on the graph with all boundary vertices identified.}
\label{forest-tree}
\end{figure}
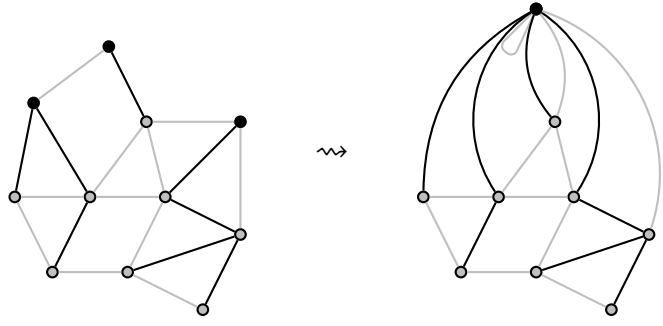

\begin{proof}[Proof of \cref{grove-determinant}]
 We start by rewriting the definition of $G_k(\mathbf{a},\mathbf{b})$ in terms of trees and paths, which will let us reduce to a lemma from \cite{Curtis-Morrow}. Each grove in $\mathrm{Grove}(\mathbf{a},\mathbf{b})$ contains a unique family of (necessarily disjoint) paths $P = \{p_i:a_i\to b_i\}_{i=1}^n$, and the remainder of the grove is a collection of trees each containing exactly one vertex in $P$ or the rest of $\bd V$; we summarize this with the set equation
 \[\mathrm{Grove}_k(\mathbf{a},\mathbf{b}) = \coprod_{P\in\mathrm{Path}_k(\mathrm{a},\mathrm{b})}\ \coprod_{g\in\mathrm{Tree}(\partial V \cup P)}(P\sqcup g).\]
 The weight of such a grove is the product of the weights of the path family and the collection of trees, so we can say that
 \begin{align*}
 G_k(\mathbf{a},\mathbf{b}) &= \frac{1}{Z_{\bd V}}\sum_{P\in\mathrm{Path}_k(\mathbf{a},\mathbf{b})}\  \sum_{g\in \mathrm{Tree}(\bd V\cup P)}\wt(P)\cdot \wt(g)\\
 &= \frac{1}{Z_{\bd V}}\sum_{P\in\mathrm{Path}_k(\mathbf{a},\mathbf{b})}\wt(P)\cdot Z_{\bd V \cup P}\\
 &= \frac{1}{\det(K_{\intr V}^{\intr V})}\sum_{P\in\mathrm{Path}_k(\mathbf{a},\mathbf{b})}\wt(P)\cdot \det(K_{\intr V\setminus P}^{\intr V\setminus P}).
 \end{align*}
 Now by \cite[Lemma 3.12]{Curtis-Morrow}, the matrix $\Lambda \coloneqq K_{\bd V}^{\bd V} - K_{\intr V}^{\bd V}(K_{\intr V}^{\intr V})^{-1}K_{\bd V}^{\intr V}$ has the following minor:
 \begin{align*}\det(\Lambda_{a_ib_j})_{i,j=1}^{k} &= \frac{(-1)^k}{\det(K_{\intr V}^{\intr V})}\sum_{\sigma\in\symgroup{k}} \sgn(\sigma)\sum_{P\in\mathrm{Path}_k(\mathbf{a},\mathbf{b}_\sigma)} \wt(P) \det(K_{\intr V\setminus P}^{\intr V\setminus P}),
 \intertext{which we can now rewrite as}
 &= (-1)^k\sum_{\sigma\in\symgroup{k}}\sgn(\sigma) G_k(\mathbf{a},\mathbf{b}_\sigma)\\
 &= (-1)^k \tilde G_k(\mathbf{a},\mathbf{b}).
\end{align*}
The case $k=1$ therefore tells us that the $a_ib_j$th entry of $\Lambda$ is $-\tilde G_1(a_i,b_j)$, so 
\[\det\bigl(\tilde G_1(a_i,b_j)\bigr)_{i,j=1}^k = \det(-\Lambda_{a_ib_j})_{i,j=1}^k = (-1)^k\det(\Lambda_{a_ib_j})_{i,j=1}^k = \tilde G_k(\mathbf{a},\mathbf{b}).\qedhere\]
\end{proof}

\begin{remark}
 For graphs embedded in a disc with boundary vertices on the boundary circle, Kenyon and Wilson show in \cite{Kenyon-Wilson} that for \emph{any} partition $\tau$ of the boundary vertices, $1/Z_{\bd V}$ times the sum of weights of all groves inducing partition $\tau$ can be expressed as a polynomial in the entries of the matrix $\Lambda$.
 If $\mathbf{a}$ and $\mathbf{b}$ are arranged so that $a_1,\ldots,a_k,b_k,\ldots,b_1$ are in counterclockwise order around the boundary circle, then $\tilde G_k(\mathbf{a},\mathbf{b}) = G_k(\mathbf{a},\mathbf{b})$ because of the planarity assumption, and Kenyon and Wilson's theorem gives us $G_k(\mathbf{a},\mathbf{b}) = (-1)^k\det(\Lambda_{a_ib_j})_{i,j=1}^k$ as a special case.
\end{remark}

Finally, now that we have a sequence of functions $\tilde G_k$ satisfying the determinant relation, we can use the determinant-to-Pfaffian principle to deduce a version of Stembridge's theorem for undirected graphs:

\begin{corollary}\label{grove-pfaffian}
 Given a finite weighted undirected graph $G$ with choice of boundary  $\bd V$, and given a partition of $\bd V$ into two ordered sets $A$ and $B$, define a family of functions $\tilde H_k: A^k\to R$ by
 \[\tilde H_k(\mathbf{a}) = \sum_{\substack{\mathbf{b}=(b_1,\dots,b_{k})\in B^{k}:\\b_1<b_2<\dots<b_{k}}}\tilde G_{k}(\mathbf{a},\mathbf{b}).\]
 Then for even $k$, we have $\tilde H_k(\mathbf{a}) = \Pf\bigl(\tilde H_2(a_i,a_j)\bigr)_{i,j=1}^k$.
\end{corollary}

%%%%%

\section{Alternating flows on planar circular networks}\label{flows}

As a final application of the determinant-to-Pfaffian principle, we now turn to the work of Kelli Talaska on alternating flows in \cite{Talaska}. In this section, $G$ is a finite planar directed graph equipped with an embedding into the closed unit disc. This embedding gives $G$ the structure of a graph-with-boundary, where the interior vertices $\intr V$ are those embedded into the interior of the disc, and the boundary vertices $\bd V$ are those embedded into the boundary circle. We assume that each boundary vertex is incident to exactly one edge, making it either a source or a sink, and denote the set of source boundary vertices by $A$ and the set of sink boundary vertices by $B$. We fix a total order on $\bd V = A\cup B$ by starting at an arbitrary boundary vertex and continuing clockwise around the boundary circle. We also weight each edge of $G$ by its corresponding formal variable in the field of rational functions $R=\Q(e: e\in\Ed(G))$. We follow Talaska in calling such an embedded weighted directed graph $G$ a \emph{planar circular network}.

An \emph{alternating flow} $f$ on $G$ is a subset of $\Ed(G)$ such that, for each interior vertex $v$, the edges incident to $v$ in $f$ alternate in orientation (toward or away from $v$) under the circular ordering given by the planar embedding of $G$; see \cref{plancircflow}.

Each flow $F$ has three associated quantities:
\begin{itemize}
\item Its \emph{weight} $\wt(f)$, which is the product of the weights of edges in $f$, as usual.
\item Its \emph{collision index} $\theta(f)$, which is the total over all interior vertices $v$ incident to edges in $f$ of $\frac12($the number of edges of $f$ incident to $v) - 1$.
\item Subsets $A'\subset A$ and $B'\subset B$ consisting of those boundary vertices incident to edges in $f$.  We say that $f$ \emph{connects} $A'$ to $B'$. A simple counting argument shows that $|A'|=|B'|$.
\end{itemize}

If $A' = \{a_1,\dots,a_k\}\subset A$ and $B' = \{b_1,\dots,b_k\}\subseteq B$ are sets of $k$ boundary sources and sinks, then we write $\mathrm{Flow}_k(A',B')$ for the set of alternating flows connecting $\A'$ to $B'$.  We also define the set $\mathrm{Cons}(G) = \mathrm{Flow}_0(\emp, \emp)$ to be the set of \emph{conservative} alternating flows, which do not use any edges incident to the boundary of $G$. (This includes the empty flow, which has weight $1$.)

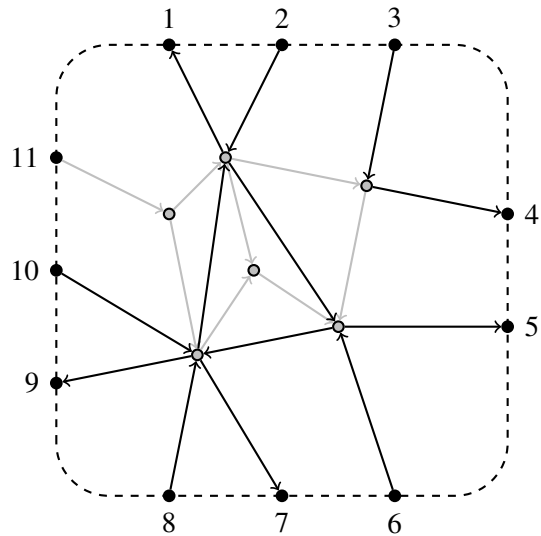
\begin{figure}[h]
\[\begin{tikzpicture}[thick, scale=1.5]
 \draw[rounded corners=20pt, dashed] (0,0) -- (0,4) -- (4,4) -- (4,0) -- cycle;
 \vertex{1}{(1,4)}{90:1, fill=black}
 \vertex{2}{(2,4)}{90:2, fill=black}
 \vertex{3}{(3,4)}{90:3, fill=black}
 \vertex{4}{(4,2.5)}{0:4, fill=black}
 \vertex{5}{(4,1.5)}{0:5, fill=black}
 \vertex{6}{(3,0)}{270:6, fill=black}
 \vertex{7}{(2,0)}{270:7, fill=black}
 \vertex{8}{(1,0)}{270:8, fill=black}
 \vertex{9}{(0,1)}{180:9, fill=black}
 \vertex{10}{(0,2)}{180:10, fill=black}
 \vertex{11}{(0,3)}{180:11, fill=black}
 \vertex{12}{(1.5,3)}{}
 \vertex{13}{(2.75,2.75)}{}
 \vertex{14}{(1,2.5)}{}
 \vertex{15}{(1.25,1.25)}{}
 \vertex{16}{(2.5,1.5)}{}
 \vertex{17}{(1.75,2)}{}
 \arc[lightgray]{11}{14}{}
 \arc[lightgray]{14}{12}{}
 \arc[lightgray]{14}{15}{}
 \arc[lightgray]{12}{13}{}
 \arc[lightgray]{13}{16}{}
 \arc[lightgray]{12}{17}{}
 \arc[lightgray]{15}{17}{}
 \arc[lightgray]{17}{16}{}
 \arc{12}{1}{}
 \arc{2}{12}{}
 \arc{3}{13}{}
 \arc{13}{4}{}
 \arc{16}{5}{}
 \arc{6}{16}{}
 \arc{15}{7}{}
 \arc{8}{15}{}
 \arc{15}{9}{}
 \arc{10}{15}{}
 \arc{15}{12}{}
 \arc{12}{16}{}
 \arc{16}{15}{}
\end{tikzpicture}
\]
\caption{\label{plancircflow}A planar circular network. The dark edges form an alternating flow $f$ connecting $\{2, 3, 6, 8, 10\}$ to $\{1, 4, 5, 7, 9\}$ with $\theta(f)=0+1+1+2=4$.}
\end{figure}

Then we define these quantities:
\begin{align*}
 C &\coloneqq \sum_{f\in\mathrm{Cons}(G)} 2^{\theta(f)} \wt(f)\\
 F_k(A',B') &\coloneqq \frac{1}{C} \sum_{f\in \mathrm{Flow}_k(A',B')} 2^{\theta(f)} \wt(f).
\end{align*}
Finally, we define the quantity $\tilde F_k(\mathbf{a},\mathbf{b})$ for tuples $\mathbf{a}\in A^k$ and $\mathbf{b}\in B^k$. If there are repeated vertices in $\mathbf{a}$ or $\mathbf{b}$, we set $\tilde F_k(\mathbf{a},\mathbf{b})=0$. Otherwise, $\tilde F_k(\mathbf{a},\mathbf{b})$ will equal $F_k(\{a_1,\dots,a_k\},\{b_1,\dots,b_k\})$ up to an overall sign that depends only on the order properties of $\mathbf{a}$ and $\mathbf{b}$. Namely, let $\mathbf{a} \frown (A\setminus\mathbf{a})$ be the concatenation of $\mathbf{a}$ with all the rest of the elements of $A$, in order, and similarly form the tuple $\mathbf{b}\frown (A\setminus\mathbf{a})$. As both are tuples of elements of the totally ordered set $\bd V$, we can count their \emph{inversions}, pairs of entries out of order with respect to the total order on $\bd V$. Then set
\[\sgn(\mathbf{a},\mathbf{b}) \coloneqq (-1)^{\ds\#\text{ of inversions in }\mathbf{a} \frown (A\setminus\mathbf{a})\text{ and }\mathbf{b}\frown (A\setminus\mathbf{a})}\]
and finally
\[
 \tilde F_k(\mathbf{a},\mathbf{b}) \coloneqq \sgn(\mathbf{a},\mathbf{b}) F_k(\{a_1,\dots,a_k\},\{b_1,\dots,b_k\}).
\]

Then the $\tilde F_k$ obey a Lindstr\"om-type determinant relation:

\begin{theorem}\label{flow-determinant}
 Let $G$ be a planar circular network with its sets $A$ and $B$ of boundary sources and sinks, respectively. For each natural number $k$ and tuples $\mathbf{a}=(a_1,\dots,a_k)\in A^k$ and $\mathbf{b}=(b_1,\dots,b_k)\in B^k$, we have
 \[\tilde F_k(\mathbf{a},\mathbf{b}) = \det\bigl(\tilde F_1(a_i,b_j)\bigr)_{i,j=1}^k.\]
\end{theorem}

The proof uses a small generalization of the main theorem of \cite{Talaska}, reproduced here in our language:

\begin{theorem}[see Definition 2.6 and Corollary 4.3 of \cite{Talaska}]\label{talaska}
 Let $G$ be a planar circular network with boundary set $\bd V$, partitioned into the subsets of boundary sources $A$ and boundary sinks $B$. There is a matrix $M$ with rows indexed by $A$ and columns indexed by $\bd V$, called the \emph{boundary measurement matrix}, such that
 \begin{itemize}
  \item The column of $M$ indexed by a boundary source $a\in \bd V$ consists of all $0$'s except for a $1$ in the row indexed by $a\in A$.
  \item If $V'$ is a set of boundary vertices with $|V'|=|A|$, then the maximal minor of $M$ using the columns corresponding to $V'$ and all the rows is equal to
  \[\det\bigl(M_{A}^{V'}\bigr) = F_k(A\setminus V',V'\setminus A),\]
  where $k=|A\setminus V'| = |V'\setminus A|$.
 \end{itemize}
\end{theorem}

(Note: In \cite{Talaska}, every flow $f$ is said to connect all of $A$ to some equal-sized subset $V'$ of $\bd V$; those boundary sources not incident to an edge of $f$ are connected to themselves. The flows that connect $A$ to $V'$ in this sense of Talaska are exactly those that connect $A\setminus V'$ to $V'\setminus A$ in the sense of this paper.)

\begin{corollary}\label{talaska-corollary}
 If $M$ is the boundary measurement matrix for a planar circular network $G$, and $\mathbf{a} = (a_1,\dots,a_k)\in A^k$ and $\mathbf{b}=(b_1,\dots,b_k)\in B^k$ are tuples of boundary sources and sinks, then
 \[\det(M_{a_ib_j})_{i,j=1}^k = \tilde F_k(\mathbf{a},\mathbf{b}).\]
\end{corollary}

\begin{proof}
 If there are any repeated vertices among the entries of $\mathbf{a}$ or $\mathbf{b}$, then both sides of the equation are zero, so in the following, assume that the entries of $\mathbf{a}$ and $\mathbf{b}$ are all distinct.
 
 Consider the submatrix of $M$ whose rows are indexed by $\mathbf{a} \frown (A\setminus\mathbf{a})$ and whose columns are indexed by $\mathbf{b}\frown (A\setminus\mathbf{a})$; it has the following block structure:
 \[\bordermatrix{ & \mathbf{b} & A\setminus\mathbf{a} \cr
 \mathbf{a} & (M_{a_ib_j})_{i,j=1}^k & 0 \cr
 A\setminus\mathbf{a} & \ast & I}\]
 On the one hand, as a block matrix its determinant equals $\det(M_{a_ib_j})_{i,j=1}^k$.
 On the other hand, sorting its rows and columns (which changes the sign of its determinant by a factor of $\sgn(\mathbf{a},\mathbf{b})$) yields exactly the submatrix $M_A^{V'}$ of $M$, where $V' = \{b_1,\dots,b_k\} \cup (A \setminus \{a_1,\dots,a_k\})$. 
 Therefore 
 \begin{align*}
  \det(M_{a_ib_j})_{i,j=1}^k &= \sgn(\mathbf{a},\mathbf{b})\det(M_A^{V'})\\
  &= \sgn(\mathbf{a},\mathbf{b}) F_k(A\setminus V', V'\setminus A)\\
  &= \sgn(\mathbf{a},\mathbf{b}) F_k(\{a_1,\dots,a_k\}, \{b_1,\dots,b_k\})\\
  &= \tilde F_k(\mathbf{a},\mathbf{b}).\qedhere
 \end{align*}
\end{proof}

Now we can prove that the $\tilde F_k$ satisfy the determinant relation:

\begin{proof}[Proof of \cref{flow-determinant}]
  Let $M$ be the boundary measurement matrix of $G$. Given $a\in A$ and $b\in B$, we can apply \cref{talaska-corollary} to the $1$-tuples $(a)$  and $(b)$ to deduce that $M_{ab}=\tilde F_1(a,b)$. Applying \cref{talaska-corollary} again to the tuples $\mathbf{a}$ and $\mathbf{b}$ thus yields \[\det(\tilde F_1(a_i,b_j))_{i,j=1}^k = \det(M_{a_ib_j})_{i,j=1}^k = \tilde F_k(\mathbf{a},\mathbf{b})\]
 as desired.
\end{proof}

Then we can immediately apply the determinant-to-Pfaffian principle to obtain

\begin{corollary}\label{flow-pfaffian}
 Let $G$ be a planar circular network with its sets $A$ and $B$ of boundary sources and sinks. Given a tuple $\mathbf{a}=(a_1,\dots,a_k)\in A^k$, define
 \[\tilde E_k(\mathbf{a}) = \sum_{\substack{\mathbf{b}=(b_1,\dots,b_{k})\in B^{k}:\\b_1<b_2<\dots<b_{k}}}\tilde F_{k}(\mathbf{a},\mathbf{b}).\]
 Then for even $k$, we have $\tilde E_k(\mathbf{a}) = \Pf\bigl(\tilde E_2(a_i,a_j)\bigr)_{i,j=1}^k$.
\end{corollary}

%%%%%

\bibliographystyle{acm}
\bibliography{reflist}

\end{document}